\documentclass[a4paper,12pt]{amsart}

\usepackage{amssymb}
\usepackage{latexsym}
\usepackage{amsfonts}
\usepackage{amsmath}
\usepackage{eucal}
\usepackage{bm}
\usepackage{bbm}
\usepackage{graphicx}
\usepackage[english]{varioref}
\usepackage[nice]{nicefrac}
\usepackage[all]{xy}
\usepackage{amsthm}

\newcommand{\supp}{\text {\rm supp}}

\newcommand{\Ad}{\text {\rm Ad}}

\newcommand{\Hom}{{\rm Hom}}

\def\i{^{-1}}
\def\ge{\geqslant}
\def\le{\leqslant}
\def\<{\langle}
\def\>{\rangle}
\def\lto{\hookrightarrow}

\def\a{\alpha}
\def\b{\beta}

\def\G{\Gamma}
\def\d{\delta}

\def\o{\omega}

\def\s{\sigma}
\def\t{\tau}

\def\l{\lambda}

\def\Om{\Omega}

\def\ZZ{\mathbb Z}
\def\NN{\mathbb N}
\def\QQ{\mathbb Q}
\def\FF{\mathbb F}
\def\RR{\mathbb R}
\def\CC{\mathbb C}

\def\kk{\mathbf k}

\def\aff{\text aff}
\def\tch{\tilde \ch}

\def\ch{\mathcal H}

\def\co{\mathcal O}

\def\tu{\tilde u}

\def\tx{\tilde x}
\def\ty{\tilde y}

\def\tW{\tilde W}
\def\tw{\tilde w}

\def\fR{\mathfrak R}
\def\fH{\mathfrak H}

\def\subset{\subseteq}

\theoremstyle{plain}
\newtheorem{thm}{Theorem}[section]
\newtheorem*{thm*}{Theorem}
 \newtheorem{prop}[thm]{Proposition}
 \newtheorem{lem}[thm]{Lemma}
 \newtheorem{cor}[thm]{Corollary}

\theoremstyle{definition}

\theoremstyle{remark}

\newtheorem*{rmk}{Remark}
\newtheorem*{claim*}{Claim}

\begin{document}

\author{Xuhua He}
\address{Department of Mathematics, University of Maryland, College Park, MD 20742, USA and department of Mathematics, HKUST, Hong Kong}
\email{xuhuahe@math.umd.edu}
\author{Sian Nie}
\address{Institute of Mathematics, Academy of Mathematics and Systems Science, Chinese Academy of Sciences, 100190, Beijing, China}
\email{niesian@amss.ac.cn}
\title{Cocenters and representations of affine $0$-Hecke algebras}

\begin{abstract}
In this paper, we study the relation between the cocenter $\overline{\tilde \ch_0}$ and the finite dimensional representations of an affine $0$-Hecke algebra $\tilde \ch_0$. As a consequence, we obtain a new criterion on the supersingular modules: a (virtual) module of $\tilde \ch_0$ is supersingular if and only if its character vanishes on the non-supersingular part of $\overline{\tilde \ch_0}$.
\end{abstract}

\maketitle

\section*{Introduction}

\subsection{} Extended affine Hecke algebras $\tilde \ch_q$ are deformations of the group algebras of extended affine Weyl groups $\tW$ (with the parameter functioin $q$). They play an important role in the study of representations of $p$-adic groups $G$.

For complex representations, Borel correspondence relates the representation of $G$ with Iwahori fixed points to representations of $\tilde \ch_q$, where $q$ is a power of the prime number $p$.

For representations in characteristic $p$ (the defining characteristic), Vign\'eras \cite{V05} relates the representations of $G$ with representations of affine $0$-Hecke algebras $\tilde \ch_0$ and its generalization, pro-$p$ Iwahori-Hecke algebras.

\subsection{} By the work of Kazhdan-Lusztig \cite{KL} and Reeder \cite{R}, the simple modules of $\tilde \ch_q$ (for $q$ nonzero and not a root of unity) are parameterized by the triple $(s, u, \phi)$, where $s$ is a semisimple element in the dual group $G^\vee$, $u$ is a unipotent element in $G^\vee$ with $s u s \i=u^q$ and $\phi$ is a local system of Springer type.

The classification of simple modules of $\tilde \ch_0$, on the other hand, looks quite different. Abe \cite{A}  gave a classification of mod-$p$ representations in terms of parabolic inductions of simple supersingular modules. Ollivier \cite{O13} and Vign\'eras \cite{V14} classified all simple  supersingular modules in terms of supersingular characters. The proof uses Bernstein presentation \cite{V14-1} and Satake-type isomorphism \cite{V14-2}.

\subsection{} In this paper, we study the cocenter $\overline{\tch_0}$ of $\tch_0$ and the trace map $Tr: \overline{\tch_0} \to R(\tilde \ch_0)_\kk^*$ induced from the natural trace pairing between the cocenter $\overline \tch_0$ and the Grothendieck group $R(\tilde \ch_0)_\kk$ of finite dimensional representations of $\tilde \ch_0$ over an arbitrary algebraically closed field $\kk$. We then use the trace map to give a basis of $R(\tilde \ch_0)$. As a consequence, we give a new proof of the classification of simple supersingular modules of $\tilde \ch_0$. 

\

In the rest of the introduction, we explain our main results and compare them with results for $\tilde \ch_q$. For simplicity, we only state the results for the case that $\tW$ is an affine Weyl group. In the body of the paper, we tackle the general case.

\subsection{} The affine $0$-Hecke algebra $\tch_0$ has a standard $\ZZ$-basis $\{T_{\tw}; \tw \in \tW\}$ subject to quadratic relations and braid relations. For the cocenter $\overline{\tilde \ch_0}$, we have the following basis Theorem.

\begin{thm}\label{basis}
The set $\{T_{\Sigma}; \Sigma \in \text{Cyc}(\tW_{\min})\}$ forms a $\ZZ$-basis of $\overline{\tilde \ch_0}$.
\end{thm}

Here $\tW_{\min}$ is the set of elements in the affine Weyl group $\tW$ that are of minimal lengths in their conjugacy classes and $\text{Cyc}(\tW_{\min})$ is the set of cyclic-shift classes in $\tW_{\min}$ (defined in \S \ref{2.1}). The element $T_{\Sigma}$ is the image of $T_{\tw}$ in $\overline{\tilde \ch_0}$ for some, or equivalently, any $\tw \in \Sigma$.

This result is obtained using some nice properties of $\tW_{\min}$ established in \cite{HN} and an idea in \cite{He15} for finite $0$-Hecke algebras.

It is interesting to compare the cocenter of $\tch_0$ with that of $\tilde \ch_q$ for $q \neq 0$. For the latter one, a similar result is obtained in \cite{HN} (for equal parameter case) and \cite{CH} (for general case). For $\tilde \ch_q$, the cocenter has a basis indexed by the set of ``strongly conjugacy classes'' of $\tW_{\min}$, which is in natural bijection with the set of conjugacy classes of $\tW$.

\subsection{} Now we move to the trace map $Tr: \overline{\tch_0} \to R(\tilde \ch_0)_\kk^*$ and discuss its application on representations of $\tch_0$.

Using parabolic induction and the basis Theorem for the cocenter, we can essentially reduce the study of the trace map to the study of the trace map for the 0-Hecke algebras of parahoric subgroups. Notice that the 0-Hecke algebra of a parahoric subgroup is a finite $0$-Hecke algebra, whose simple modules have been classified in \cite{No}. We have

\begin{thm} \label{main'}
The set $\{\pi_{J, \G, \chi}; (J, \G) \in \aleph/\sim, \chi \in \Omega_J(\G)^\vee\}$ is a $\ZZ$-basis of $R(\tilde \ch_0)_\kk$.
\end{thm}

Here $\pi_{J, \G, \chi}$ is, roughly speaking, an $\tch_0$-module induced from certain simple module of the parabolic subalgebra $\tilde \ch_{J, 0}^+$, which is indexed by the character $\chi$ and the parahoric subalgebra of $\tilde \ch_{J, 0}^+$ of type $\G$. We refer to \S \ref{constr} for the precise definition.

\subsection{}
By combining Theorem \ref{main'} with the character formula (Theorem \ref{char}), we obtain in Proposition \ref{5.4} a new proof of the classification of simple supersingular modules. We also obtain the following criterion of supersingular modules.

\begin{thm}\label{ss}
An element $\pi \in R(\tilde \ch_0)_\kk$ is supersingular if and only if $Tr(h, \pi)=0$ for all $h \in \overline{\tilde \ch_0}^{nss}$.
\end{thm}

Here $\overline{\tilde \ch_0}^{nss}$ is the non-supersingular part of the cocenter, defined as the subspace of $\overline{\tilde \ch_0}$ spanned by $T_{\Sigma}$ and ${}^{\iota} T_{\Sigma}$, where $\Sigma$ is not contained in any proper parahoric group of $\tW$ and $\iota$ is an involution of $\tilde \ch_0$ defined in \S \ref{1.3}.




\subsection{} Again, it is interesting to compare the above results on $R(\tch_0)_\kk$ to the results on $R(\tch_q)_\CC$ for generic $q \neq 0$.

The trace pairing $Tr: \overline{\tch_q} \to R(\tch_q)_\CC^*$ and the cocenter-representation duality for $\tch_q$ are studied in \cite{CH}. Using the parabolic induction, we are reduced to study the trace pairing between the so-called rigid cocenter and the rigid modules. Here the rigid cocenter is the subspace of $\overline{\tch_q}$ spanned by the images all proper parahoric subalgebras. The rigid modules are constructed using Lusztig's reduction theorem from affine Hecke algebras to graded affine Hecke algebras, and Springer representations for the finite Weyl group in the corresponding graded affine Hecke algebras. It is proved in \cite[Theorem 1.1]{CH} that such pairing is perfect.


For $\tilde \ch_0$, as we have seen above, the situation is different. What appears in this situation is not the representations of finite Weyl groups (or equivalently, the finite Hecke algebras with generic parameters), but that of the finite $0$-Hecke algebra instead. This provides an interpretation for the difference between the representation theory of $\tch_q$ for $q \neq 0$ and that of $\tch_0$.


\subsection{} The paper is organized as follows.

In section 1, we recall the definition of affine $0$-Hecke algebras, parabolic algebras, and trace maps. In section 2, we describe the cocenters of extended affine $0$-Hecke algebras. In section 3, we introduce the standard pairs and use them to compute the characters of $\tch_0$-modules. In section 4, we construct some finite-dimensional modules and provide some character formulas. In section 5, we give a basis of the Grothendieck group of finite dimensional modules and study rigid and supersingular modules.

\section{Preliminary}

\subsection{} Let $\fR=(X, R, Y, R^\vee, F_0)$ be a based root datum, where $X$ and $Y$ are free abelian groups of finite rank together with a perfect pairing $\< , \>: X \times Y \to \ZZ$, $R \subset X$ is the set of roots, $R^\vee \subset Y$ is the set of coroots and $F_0 \subset R$ is the set of simple roots. Let $\a \mapsto \a^\vee$ be the natural bijection from $R$ to $R^\vee$ such that $\<\a, \a^\vee\>=2$. For $\a \in R$, we denote by $s_\a: X \to X$ the corresponding reflections stabilizing $R$. Let $R^+ \subset R$ be the set of positive roots determined by $F_0$. Let $X^+=\{\l \in X; \<\l, \a^\vee\> \ge 0, \, \forall \a \in R^+\}$. For any $v \in X_\QQ$ , we set $J_v=\{\a \in F_0; \<v, \a^\vee\>=0\}$. For any $J \subset F_0$, we set $X^+(J)=\{\l \in X^+; J_\l=J\}$.

\subsection{} Let $W_0$ be the (finite) Weyl group generated by the set of simple reflections $S_0=\{s_\a; \a \in F_0\}$.

Let $W_{\aff}=\ZZ R \rtimes W_0$ be the affine Weyl group and $S_{\aff} \supset S_0$ be the set of simple reflections in $W_{\aff}$. Then $(W_{\aff}, S_{\aff})$ is a Coxeter group. Let $\tW=X \rtimes W_0$ be the extended affine Weyl group. Then $W_{\aff}$ is a subgroup of $\tW$. For $\l \in X$, we denote by $t^\l \in \tW$ the corresponding translation element.

Let $V=X \otimes_\ZZ \RR$. For $\a \in R$ and $k \in \ZZ$, set $$H_{\a, k}=\{v \in V; \<v, \a^\vee\>=k\}.$$ Let $\fH=\{H_{\a, k}; \a \in R, k \in \ZZ\}$.  Connected components of $V-\cup_{H \in \fH}H$ are called alcoves. Let $$C_0=\{v \in V; 0 < \<v, \a^\vee\> <1, \, \forall \a \in R^+\}$$ be the fundamental alcove.
We may regard $W_{\aff}$ and $\tW$ as subgroups of affine transformations of $V$, where $t^\l$ acts by translation $v \mapsto v+\l$ on $V$. The actions of $W_{\aff}$ and $\tW$ on $V$ preserve the set of alcoves.

For any $\tw \in \tW$, we denote by $\ell(\tw)$ the number of hyperplanes in $\fH$ separating $C_0$ from $\tw C_0$. Then $\tW=W_{\aff} \rtimes \Omega$, where $\Omega=\{\tw \in \tW; \ell(\tw)=0\}$ is the subgroup of $\tW$ stabilizing fundamental alcove $C_0$. The conjugation action of $\Omega$ on $\tW$ preserves the set $S_{\aff}$ of simple reflections in $W_{\aff}$. 

For any $x \in W_{\aff}$ and any $\t \in \Omega$, we define $$\supp(x \t)=\cup_{i \in \NN} \t^i(\supp(x)) \t^{-i}.$$ Here $\supp(x)$ is the set of simple reflections that appear in some (or equivalently, any) reduced expression of $x$.

\subsection{}\label{1.3} The (generic) Hecke algebra $\tilde \ch_q$ associated to the extended affine Weyl group $\tW$ is an associative $\ZZ[q]$-algebra with basis $\{T_{\tw}; \tw \in \tW\}$ subject to the following relations \begin{gather*} T_{\tx} T_{\ty}=T_{\tx \ty}, \quad \text{ if } \ell(\tx)+\ell(\ty)=\ell(\tx \ty); \\ (T_s+1)(T_s-q)=0, \quad \text{ for } s \in S_{\aff}. \end{gather*}

If we set $q=0$, then the second relation becomes $T_s^2=-T_s$ and the $\ZZ$-algebra we obtain is called the (affine) $0$-Hecke algebra associated to $\tW$. We denote it by $\tilde \ch_0$.

By \cite[Corollary 2]{V05}, the map $T_{\tw} \mapsto {}^\iota T_{\tw}:=(-q)^{\ell(\tw)} T_{\tw\i}\i$ gives an involution $\iota$ of $\tch_q$. We still denoted by $\iota$ the induced involution of $\tch_0$.



\subsection{} Let $[\tilde \ch_0, \tilde \ch_0]$ be the commutator of $\tilde \ch_0$, the $\ZZ$-submodule spanned by $[T_{\tx}, T_{\ty}]:=T_{\tx} T_{\ty}-T_{\ty} T_{\tx}$ for $\tx, \ty \in \tW$. Let $\overline{\tilde \ch_0}=\tilde \ch_0/[\tilde \ch_0, \tilde \ch_0]$ be the cocenter of $\tch_0$. Denote by $R(\tilde \ch_0)_\kk$ the Grothendieck group of finite dimensional representations of $\tilde \ch_0$ over an arbitrary algebraically closed field $\kk$. Consider the trace map $$Tr: \overline{\tilde \ch_0} \to R(\tilde \ch_0)_\kk^*, \qquad h \mapsto (V \mapsto Tr(h, V)).$$

Similar map for generic $q \in \CC^\times$ and $\kk=\CC$ is studied in the joint work of Ciubotaru and the first-named author \cite{CH}, in which case the trace map is injective and there is a ``perfect pairing'' between the rigid-cocenter and rigid-representations of $\tilde \ch_q$.

For $q=0$, the situation is different. The map is not injective. However, there is still a nice pairing between cocenter and representations.

\subsection{} \label{parabolic} Now we introduce parabolic subalgebras.

For any $J \subset F_0$, we denote by $R_J$ the set of roots spanned by $J$ and set $R^\vee_J=\{a^\vee; \a \in R_J\}$. Let $\fR_J=(X, R_J, Y, R^\vee_J, J)$ be the based root datum corresponding to $J$. Let $W_J \subset W_0$ and $\tW_J=X \rtimes W_J$ be the Weyl group and the extended affine Weyl group of $\fR_J$ respectively. We say $\tw \in \tW_J$ is $J$-positive if $\tw \in  t^\l W_J $ for some $\l \in X$ such that $\<\l, \a\> \ge 0$ for $\a \in R^+ \smallsetminus R_J$. Denote by $\tW_J^+$ the set of $J$-positive elements, which is a submonoid of $\tW_J$, see \cite[Section 6]{BK} and \cite[II.4]{V98}.

We set $\fH_J=\{H_{\a, k} \in \fH; \a \in R_J, k \in \ZZ\}$ and $C_J=\{v \in V; 0 < \<v, \a^\vee\> < 1, \a \in R_J^+\}$. For any $\tw \in \tW_J$, we denote by $\ell_J(\tw)$ the number of hyperplanes in $\fH_J$ separating $C_J$ from $\tw C_J$.

Let $\tilde \ch_{J, 0}$ be the affine $0$-Hecke algebra associated to $\fR_J$ with standard basis $T_{\tw}^J$ for $\tw \in \tW_J$. Let $\tilde \ch_{J, 0}^+$ be the subalgebra of $\tilde \ch_{J, 0}$ spanned by $T_{\tw}^J$ for $\tw \in \tW_J^+$. We have a natural embedding $$\tilde \ch_{J, 0}^+ \lto \tilde \ch_0, \qquad T_{\tw}^J \mapsto T_{\tw}.$$ Notice that this embedding does not extend to an algebra homomorphism $\tilde \ch_{J, 0} \to \tilde \ch_0$ since $T_{t^\l}^J$ for $\l \in X^+(J)$ is invertible in $\tilde \ch_{J, 0}$, but $T_{t^\l}$ is not invertible in $\tilde \ch_0$ unless $J=F_0$.

Let $(W_J)_{\aff}=\ZZ R_J \rtimes W_J$ and $J_{\aff} \supseteq J$ the set of simple reflections of $(W_J)_{\aff}$. Then $\tW_J=(W_J)_{\aff} \rtimes \Omega_J$, where $\Omega_J=\{\tw \in \tW_J; \ell_J(\tw)=0\}$. We denote by $\ch_{J, 0}$ the $0$-Hecke algebra associated to $(W_J)_{\aff}$.

We denote by $\tW^J$ (resp. ${}^J \tW$) the set of minimal coset representatives in $\tW/W_J$ (resp. $W_J \setminus \tW$). For $J, K \subset F_0$, we simply write $\tW^J \cap {}^K \tW$ as ${}^K \tW^J$. We define ${}^J W_0, W_0^J$ and ${}^J W_0^K$ in a similar way.

\section{Cocenter of $\tilde \ch_0$}

\subsection{}\label{2.1} For $\tw, \tw' \in \tW$ and $s \in S_{\aff}$, we write $\tw \xrightarrow{s} \tw'$ if $\tw'=s \tw s$ and $\ell(\tw') \le \ell(\tw)$.  We write $\tw \to \tw'$ if there exists a sequence $\tw=\tw_0, \tw_1, \cdots, \tw_n=\tw'$ of elements in $\tW$ such that for any $k$, $\tw_{k-1} \xrightarrow{s_k} \tw_k$ for some $s_k \in S_{\aff}$. We write $\tw \tilde \approx \tw'$ if there exists $\t \in \Omega$ such that $\tw \to \t \tw' \t \i$ and $\t \tw' \t \i \to \tw$ and we say that $\tw$ and $\tw'$ are in the same cyclic-shift class.


Note that $\tilde \approx$ is an equivalence relation. Let $cl(\tW)$ be the set of conjugacy classes of $\tW$. For any $\co \in cl(W)$, let $\co_{\min}$ be the set of minimal length elements in $\co$. Since $\tilde \approx$ is compatible with the length function, $\co_{\min}$ is a union of cyclic-shift classes.

Let $\tW_{\min}=\sqcup_{\co \in cl(W)} \co_{\min}$ and $\text{Cyc}(\tW_{\min})$ the set of cyclic-shift classes in $W_{\min}$.








\subsection{} Now we introduce a partial order on $\text{Cyc}(\tW_{\min})$.

Let $w \in \tW$ and $\Sigma \in \text{Cyc}(\tW_{\min})$. We write $\Sigma \preceq \tw$ if there exists $\tw' \in \Sigma$ such that $\tw' \le \tw$.

For $\Sigma, \Sigma' \in \text{Cyc}(\tW_{\min})$, we write $\Sigma' \preceq \Sigma$ if $\Sigma' \preceq \tw$ for some $\tw \in \Sigma$. By \cite[Corollary 4.6]{He05}, $\Sigma' \preceq \Sigma$ if and only if $\Sigma' \preceq \tw$ for any $\tw \in \Sigma$. In particular, $\preceq$ is transitive, which defines a partial order on $\text{Cyc}(\tW_{\min})$.

We have the following result.

\begin{prop}\label{Sigma}
Let $\tw \in \tW$. Then

(1) The set $\{\Sigma \in \text{Cyc}(\tW_{\min}); \Sigma \preceq \tw\}$ contains a unique maximal element $\Sigma_{\tw}$.

(2) Let $s \in S_{\aff}$ such that $\tw \to s \tw s$. Then $$\Sigma_{\tw}=\begin{cases} \Sigma_{s \tw s}, & \text{ if } \ell(s \tw s)=\ell(\tw); \\ \Sigma_{s \tw}, & \text{ if } \ell(s \tw s)<\ell(\tw). \end{cases}$$
\end{prop}

A similar statement is proved in \cite[Proposition 6.2 (1)]{He15} for finite Weyl groups. The same proof also works for extended affine Weyl groups.

\

We also have the following result, which follows directly from the definition of $\Sigma_{\tw}$.

\begin{lem}\label{tau-w}
Let $\tw \in \tW$ and $\t \in \Omega$. Then $\Sigma_{\tw}=\Sigma_{\t \tw \t \i}$.
\end{lem}

\subsection{} By definition, if $\tw \tilde \approx \tw'$, then the images of $T_{\tw}$ and $T_{\tw'}$ in $\overline{\tilde \ch_0}$ are the same. In particular, for any $\Sigma \in \text{Cyc}(\tW_{\min})$, we denote by $T_{\Sigma}$ the image of $T_{\tw}$ in $\overline{\tilde \ch_0}$ for any $\tw \in \Sigma$. We also denote by $\ell(\Sigma)$ the length of any element in $\Sigma$.

Similar to the proof of \cite[Proposition 6.2 (2)]{He15}, we have that

\begin{prop}\label{tw-sigma}
Let $\tw \in \tW$. Then the image of $T_{\tw}$ in $\overline{\tilde \ch_0}$ equals $(-1)^{\ell(w)-\ell(\Sigma_{\tw})} T_{\Sigma_{\tw}}$.
\end{prop}

We also need the following observation on the commutator of $\tilde \ch_0$.

\begin{lem}\label{red-lem}
The $\ZZ$-module $[\tilde \ch_0, \tilde \ch_0]$ is spanned by $[T_{\tw}, T_x]$ for $\tw \in \tW$ and $x \in S_{\aff} \cup \Omega$.
\end{lem}

\begin{proof}
Let $[\tilde \ch_0, \tilde \ch_0]'$ be the submodule of $[\tilde \ch_0, \tilde \ch_0]$ spanned by $[T_{\tw}, T_x]$ for $x \in S_{\aff} \cup \Omega$. It suffices to show that $[T_{\tw}, T_{\tw'}] \in [\tilde \ch_0, \tilde \ch_0]'$ for any $\tw, \tw' \in \tW$.

We argue by induction on $\ell(\tw')$. If $\ell(\tw')=0$, then it follows by definition. Let $k \ge 1$. Suppose that $[T_{\tw}, T_{\tu}] \in [\tilde \ch_0, \tilde \ch_0]'$ for any $\tu$ with $\ell(\tu)<k$. Let $s \in S_{\aff}$ with $s \tw'<\tw'$. Then $$[T_{\tw}, T_{\tw'}]=[T_{\tw} T_s, T_{s \tw'}]+[T_{s \tw'} T_{\tw}, T_s].$$ By inductive hypothesis, $[T_{\tw}, T_{\tw'}] \in [\tilde \ch_0, \tilde \ch_0]'$.
\end{proof}

\subsection{} Now we prove Theorem \ref{basis}.

Let $M$ be the free $\ZZ$-module with basis $\{[\Sigma]; \Sigma \in {Cyc}(\tW_{\min})\}$. Define a $\ZZ$-linear map $$\psi: \tilde \ch_0 \to M, \qquad T_{\tw} \mapsto (-1)^{\ell(w)-\ell(\Sigma_{\tw})} [\Sigma_{\tw}].$$

Let $\tw \in \tW$ and $s \in S_{\aff}$. We show that

(a) $[T_{\tw}, T_s] \in \ker \psi$.

If $\tw<\tw s$ and $\tw<s \tw$, then $\ell(\tw \s)=\ell(s \tw)$ and $[T_{\tw}, T_s]=T_{\tw s}-T_{s \tw}$. By Proposition \ref{Sigma} (2), $\Sigma_{\tw s}=\Sigma_{s \tw}$ and $\psi([T_{\tw}, T_s])=(-1)^{\ell(\tw)+1-\ell(\Sigma_{\tw s})} \bigl(\Sigma_{\tw s}-\Sigma_{s \tw} \bigr)=0$.

If $\tw s, s \tw<\tw$, then $[T_{\tw}, T_s]=0 \in \ker \psi$.

If $s \tw<\tw<\tw s$, then $[T_{\tw}, T_s]=T_{\tw s}+T_{\tw}$. By Proposition \ref{Sigma}(2),  $\Sigma_{\tw s}=\Sigma_{\tw}$ and $\psi([T_{\tw}, T_s])=(-1)^{\ell(\tw)+1-\ell(\Sigma_{\tw s})} \Sigma_{\tw s}+(-1)^{\ell(\tw)-\ell(\Sigma_{\tw})} \Sigma_{\tw}=0$.

If $\tw s<\tw<s \tw$, then $[T_{\tw}, T_s]=-T_{\tw}-T_{s \tw}$ and by Proposition \ref{Sigma}(2), $\Sigma_{\tw}=\Sigma_{s \tw}$ and $\psi([T_{\tw}, T_s])=(-1)^{\ell(\tw)+1-\ell(\Sigma_{\tw})} \Sigma_{\tw}+(-1)^{\ell(\tw)+2-\ell(\Sigma_{\tw})} \Sigma_{s \tw}=0$.

Thus (a) is proved.

By Lemma \ref{tau-w}, $[T_{\tw}, T_\t] \in \ker \psi$ for any $\tw \in \tW$ and $\t \in \Omega$. Thanks to Lemma \ref{red-lem}, $[\tilde \ch_0, \tilde \ch_0] \subset \ker \psi$ and we have an induced map $\overline{\tilde \ch_0} \to M$, which we still denote by $\psi$.

On the other hand, we have a well-defined $\ZZ$-linear map $\phi: M \to \overline{\tilde \ch_0}$ which sends $[\Sigma]$ to $T_{\Sigma}$. It is easy to see that $\psi \circ \phi$ is the identity map. In particular, $\phi$ is injective. By Proposition \ref{tw-sigma}, $\phi$ is also surjective. Thus $\phi$ is an isomorphism.

\section{Standard pairs}\label{standard}

\subsection{} Let $n_0=\sharp W_0$. For any $\tw \in \tW$, $\tw^{n_0}=t^\l$ for some $\l \in X$. We set $\nu_{\tw}=\l/n_0 \in X_\QQ$ and $\bar \nu_{\tw} \in X_\QQ^+$ the unique dominant element in the $W_0$-orbit of $\nu_{\tw}$. It is easy to see that the map $\tW \to V, \tw \mapsto \bar \nu_{\tw}$ is constant on each conjugacy class of $\tW$. 

We say that an element $\tw \in \tW$ is straight if $\ell(\tw^n)=n \ell(\tw)$ for any $n \in \NN$. By \cite[Lemma 1.1]{He00}, $\tw$ is straight if and only if $\ell(\tw)=\<\bar \nu_{\tw}, 2 \rho^\vee\>$, where $\rho$ is the half sum of positive coroots. A conjugacy class that contains a straight element is called a straight conjugacy class.

It is proved in \cite[Proposition 2.8]{HN} that for each cyclic-shift class in $\tW_{\min}$, we have some representatives as follows.

\begin{prop}\label{rep}
For any $\tw \in \tW_{\min}$, there exists a subset $K \subset S_{\aff}$ with $W_K$ finite, a straight element $y \in {}^K \tW {}^K$ with $y K y \i=K$ and an element $w \in W_K$ such that $\tw \tilde \approx w y$. Here $W_K \subset W_{\aff}$ denotes the subgroup generated by reflections of $K$.
\end{prop}

\subsection{} In the situation of Proposition \ref{rep}, we call $w y$ a standard representative of the cyclic-shift class of $\tw$. By \cite[Proposition 2.2]{He00}, $\bar \nu_{\tw}=\bar \nu_{w y}=\bar \nu_y$. The expression of standard representative relates each conjugacy class of $\tW$ with a straight conjugacy class. It plays an important role in the study of combinatorial properties of conjugacy classes of affine Weyl groups \cite{HN}, $\s$-conjugacy classes of $p$-adic groups \cite{He99} and representations of affine Hecke algebras with nonzero parameters \cite{CH}.

However, for a given cyclic-shift class in $\tW_{\min}$, the standard representatives are in general, not unique. This leads to some difficulty in understanding the cyclic-shift classes in $\tW_{\min}$ and their relations to the representations of $\tilde \ch_0$.

\subsection{} To overcome the difficulty, we introduce the standard pairs as follows.

Let $w y$ be a standard representative as above. Then the conjugation by $y$ sends simple reflections in $\supp(w)$ to simple reflections. Set $K=\cup_{i \in \NN} y^i \supp(w) y^{-i}$. It is easy to see that $K$ is the smallest subset of $S_{\aff}$ that $y K y \i=K$ and $y \in {}^K \tW {}^K$.

Set $J=J_{\bar \nu_y}$. Let $z \in {}^J W_0$ with $z(\nu_y)=\bar \nu_y$. Set $x=z y z \i$ and $\G=z K z \i$. Then $\G \subset J_{\aff}$ by noticing that $z C_0 \subset C_J$ (see \S 1.2 and \S1.5).

It is easy to see that $\nu_x=\bar \nu_y \in X^+_\QQ$, $\sharp W_{\G} < +\infty$ and $x \G x\i=\G$. We say that $(x, \G)$ is a standard pair associated to (the cyclic-shift class of) $\tw$.

\begin{rmk}
There might be more than one standard pairs associated to a given cyclic-shift class. However, we will see by Theorem \ref{char} that all these standard pairs are equivalent. Here we say two standard pairs $(x, \G)$ and $(x', \G')$ are equivalent if $x=x'$ and $\G'=\o \G \o\i$ for some $\o \in \Om_{J_{\nu_x}}$.
\end{rmk}

\begin{lem}\label{wK}
Let $w y$ be a standard representative and $K=\cup_{i \in \NN} y^i \supp(w) y^{-i}$. Then for $n \gg 0$, $$T_{w y}^n=(-1)^{n \ell(w)-\ell(w_K)} T_{w_K y^n},$$ where $w_K$ is the maximal element in $W_K$.
\end{lem}

\begin{rmk}
Note that $w_K y^n \neq (w y)^n$. However, we may regard $w_K y^n$ as the $n$-th Demazure product of $w y$.
\end{rmk}

\begin{proof}
Let $\d$ be the automorphism on $W_K$ induced by the conjugation action of $y$. Let $m$ be the order of the element $w \d$ in $W_K \rtimes \<\d\>$. By \cite[Corollary 5.5]{HN0}, $T_w \d(T_w) \cdots \d^{m-1}(T_w)=T_{w_K} T_{w_1} \cdots T_{w_l}$ for some $w_1, \cdots, w_l \in W_K$ with $\ell(w_1)+\cdots+\ell(w_l)=m \ell(w)-\ell(w_K)$. Thus for any $n \ge m$, \begin{align*} T_w \d(T_w) \cdots \d^{n-1}(T_w) &=T_{w_K} T_{w_1} \cdots T_{w_l} \d^m(T_w) \cdots \d^{n-1}(T_w) \\ &=(-1)^{\ell(w_1)+\ell(w_l)+(n-m)\ell(w)} T_{w_K} \\ &=(-1)^{n \ell(w)-\ell(w_K)} T_{w_K}.
\end{align*} Here the second equality follows from the definition of $0$-Hecke algebras (as $T_{w_K} T_x=(-1)^{\ell(x)} T_{w_K}$ for any $x \in W_K$). Since $y$ is a straight element,
$T_{w y}^n=T_w \d(T_w) \cdots \d^{n-1}(T_w) T_y^n=(-1)^{n \ell(w)-\ell(w_K)} T_{w_K} T_y^n=(-1)^{n \ell(w)-\ell(w_K)} T_{w_K y^n}.$ \end{proof}

\

The following result is a variation of the length formula in \cite{IM}.

\begin{lem}
For $w \in W_0$ and $\a \in R$, set $$\d_w(\a)=\begin{cases} 0, & \text{ if } w \a \in R^+; \\ 1, & \text{ if } w \a \in R^-. \end{cases}$$ Then for any $x, y \in W_0$ and $\mu \in X$, we have that $$\ell(x t^\mu y)=\sum_{\a \in R^+} |\<\mu, \a^\vee\>+\d_x(\a)-\d_{y \i}(\a)|.$$
\end{lem}

\begin{prop}\label{u-x}
Let $(x, \G)$ be a standard pair. Then

(1) for $n \gg 0$ and $u \in {}^{J_{\nu_x}} W_0$, $\ell(u \i w_\G x^n  u)=\ell(w_\G x^n)$.

(2) for $n \gg 0$, $\ell(w_\G x^{n+n_0})=\ell(w_\G x^n)+\ell(x^{n_0})$, where $n_0=\sharp W_0$.

Here $w_\G \in W_\G \subset (W_J)_{\aff}$ is the unique element with maximal length with respect to $\ell_J$.
\end{prop}

\begin{proof}
Set $J=J_{\nu_x}$. We have $w_\G x^n=t^\l w$ for some $\l \in X$ and $w \in W_J$. Since $\<\nu_x, \a^\vee\>>0$ for any $\a \in F_0 \smallsetminus J$, we have $\<\l, \a^\vee\>>0$ for any $\a \in R^+ \smallsetminus R^+_J$ as $n \gg 0$.

Notice that for $\a \in R_J$, $\d_{u \i(\a)}=\d_{\a}$. Now \begin{align*} & \ell(u \i w_\G x^n u)=\sum_{\a \in R^+} |\<\l, \a^\vee\>+\d_{u \i}(\a)-\d_{u \i w \i}(\a)| \\ &=\sum_{\a \in R^+_J} |\<\l, \a^\vee\>-\d_{w \i}(\a)|+\sum_{\a \in R^+ \smallsetminus R^+_J} |\<\l, \a^\vee\>+\d_{u \i}(\a)-\d_{u \i w \i}(\a)| \\ &=\sum_{\a \in R^+_J} |\<\l, \a^\vee\>-\d_{w \i}(\a)|+\sum_{\a \in R^+ \smallsetminus R^+_J} \bigl(\<\l, \a^\vee\>+\d_{u \i}(\a)-\d_{u \i w \i}(\a) \bigr) \\ &=\sum_{\a \in R^+_J} |\<\l, \a^\vee\>-\d_{w \i}(\a)|+\sum_{\a \in R^+ \smallsetminus R^+_J} \<\l, \a^\vee\>+\sharp\{\a \in R^+ \smallsetminus R^+_J, u \i(\a) \in R^-\} \\ & \qquad-\sharp\{\a \in R^+\smallsetminus R^+_J, u \i w \i(\a) \in R^-\} \\ &=\sum_{\a \in R^+_J} |\<\l, \a^\vee\>-\d_{w \i}(\a)|+\sum_{\a \in R^+ \smallsetminus R^+_J} \<\l, \a^\vee\>+\ell(u)-\ell(u) \\ &=\sum_{\a \in R^+_J} |\<\l, \a^\vee\>-\d_{w \i}(\a)|+\sum_{\a \in R^+ \smallsetminus R^+_J} \<\l, \a^\vee\>.
\end{align*}

This proves part (1).

For part (2), \begin{align*} & \ell(w_\G x^{n+n_0})=\sum_{\a \in R^+} |\<\l+n_0 \nu_x, \a^\vee\>-\d_{w \i}(\a)| \\ &=\sum_{\a \in R^+_J} |\<\l+n_0 \nu_x, \a^\vee\>-\d_{w \i}(\a)|+\sum_{\a \in R^+\smallsetminus R^+_J} \<\l+n_0 \nu_x, \a^\vee\> \\ &=\sum_{\a \in R^+_J} |\<\l, \a^\vee\>-\d_{w \i}(\a)|+\sum_{\a \in R^+\smallsetminus R^+_J} \<\l, \a^\vee\>+\sum_{\a \in R^+\smallsetminus R^+_J} \<n_0 \nu_x, \a^\vee\> \\ &=\ell(w_\G x^n)+\ell(x^{n_0}).
\end{align*}
\end{proof}

\

As a consequence, we have

\begin{cor}\label{xG}
Let $(x, \G)$ be a standard pair associated to the standard representative $w y$ and $K=\cup_{i \in \NN} y^i \supp(w) y^{-i}$. Then for $n \gg 0$, $w_\G x^n \tilde \approx w_K y^n$.
\end{cor}

\begin{proof}
Suppose that $z=s_1 \cdots s_k$ for $s_1, \cdots, s_k \in S_0$. Set $z_i=s_1 \cdots s_i$ for $1 \le i \le k$. Then $z_i \in {}^{J_{\nu_x}} W_0$. By Proposition \ref{u-x} (1), $\ell(z_i \i w_\G x^n z_i)=\ell(z_{i+1} \i w_\G x^n z_{i+1})$ for $0 \le i \le k-1$. Hence $z_i \i w_\G x^n z_i \tilde \approx z_{i+1} \i w_\G x^n z_{i+1}$ for $0 \le i \le k-1$. Therefore $w_\G x^n \tilde \approx z \i w_\G x^n z=w_K y^n$.
\end{proof}

\

Now we show that the character of $T_{\tw}$ for $\tw \in \tW_{\min}$ is determined by standard pairs associated to $\tw$.

\begin{prop}\label{power}
Let $\tw \in \tW_{\min}$ and $(x, \G)$ be a standard pair associated to $\tw$. Then for $n \gg 0$, $$Tr(T_{\tw}^n, \pi)=(-1)^{n \ell(\tw)-n \ell(x)-\ell(w_\G)} Tr(T_{w_\G x^n}, \pi)$$
for any $\pi \in R(\tilde \ch_0)_\kk$.
\end{prop}

\begin{proof}
Let $w y$ be a standard representative of $\tw$. Then $w y \tilde \approx \tw$ and by definition, $T_{w y}^m-T_{\tw}^m \in [\tilde \ch_0, \tilde \ch_0]$ for any $m \ge 0$. By Lemma \ref{wK} and Lemma \ref{xG}, for $n \gg 0$, $$T_{w y}^n=(-1)^{n \ell(w)-\ell(w_K)} T_{w_K y^n} \in (-1)^{n \ell(w)-\ell(w_K)} T_{w_\G x^n}+[\tilde \ch_0, \tilde \ch_0].$$

Notice that $\ell(w)=\ell(\tw)-\ell(y) \equiv \ell(\tw)-\ell(x) \mod 2$ and $\ell(w_K) \equiv \ell(w_\G) \mod 2$. Thus $T_{\tw}^n \in (-1)^{n \ell(\tw)-n \ell(x)-\ell(w_\G)} T_{w_\G x^n}+[\tilde \ch_0, \tilde \ch_0]$.
\end{proof}

\begin{cor}\label{trace}
Let $\tw, \tw' \in \tW_{\min}$ such that there is a common standard pair associated to them. Then $Tr(T_{\tw}, \pi)=Tr(T_{\tw'}, \pi)$ for any $\pi \in R(\tilde \ch_0)_\kk$.
\end{cor}

\begin{rmk}
Notice that elements in different conjugacy classes may have the same standard pair.
\end{rmk}

\begin{proof}
By Proposition \ref{power}, $Tr(T_{\tw}^n, \pi)=Tr(T_{\tw'}^n, \pi)$ for $n \gg 0$. Thus the action of $T_{\tw}$ and $T_{\tw'}$ on $\pi$ have the same generalized eigenvalues with the same multiplications. Therefore $Tr(T_{\tw}, \pi)=Tr(T_{\tw'}, \pi)$.
\end{proof}

\section{Character formulas}

\subsection{} Let $M \in R(\tilde \ch_0)_\kk$. For any $J \subset F_0$, we set $M_J=\cap_{\l \in X^+(J)} T_{t^\l} M$. Since $M$ is a linear combination of finite dimensional vector spaces, there exists $\mu \in X^+(J)$ such that $M_J=T_{t^\mu} M$. Moreover, since the action of $T_{t^\l}$ on $M_J$ is invertible for any $\l \in X^+(J)$, we may regard $M_J$ as an $\tilde \ch_{J, 0}$-module. For $\G \subset J_{\aff}$, let $\Omega_J(\G)=\{\t \in \Omega_J; \t \G \t \i=\G\}$ and $M_{J, \G}=T^J_{w_\G} M_J$. Then $M_{J, \G}$ is an $\Omega_J(\G)$-module.

\begin{lem}\label{4.1}
Let $\tw \in \tW_{\min}$ with an associated standard pair $(x, \G)$ and $M \in R(\tilde \ch_0)_\kk$. Then for $n \gg 0$, $$Tr(T_{\tw}^n, M)=(-1)^{n \ell(\tw)-n \ell(x)} Tr((T_x^{J_{\nu_x}})^n, M_{J_{\nu_x}, \G}).$$

In particular, $Tr(T_{\tw}, M)=(-1)^{\ell(\tw)-\ell(x)} Tr(T_x^{J_{\nu_x}}, M_{J_{\nu_x}, \G})$.
\end{lem}

\begin{proof}
Set $J=J_{\nu_x}$. Let $\mu \in X^+(J)$ with $M_J=T_{t^\mu} M$. Notice that $n_0 \nu_x \in X^+(J)$, where $n_0=\sharp W_0$. There exists $m \in \NN$ such that $m n_0 \nu_x-\mu \in X^+(J)$. By Proposition \ref{u-x} (2), for $n \gg 0$, $\ell(w_\G x^{n+m n_0})=\ell(w_\G x^n)+\ell(t^{m n_0 \nu_x})=\ell(w_\G x^n)+\ell(t^{m n_0 \nu_x-\mu})+\ell(t^{\mu})$ and $$T_{w_\G x^{n+m n_0}}=T_{w_\G x^n} T_{t^{m n_0 \nu_x-\mu}} T_{t^{\mu}}.$$ Moreover, for $n \gg 0$, $w_\G x^{n+m n_0} \in \tW_J^+$ and $T_{w_\G x^{n+m n_0}}=T^J_{w_\G x^{n+m n_0}}$. Since $0 \to \ker(T_{t^{\mu}}: M \to M) \to M \to M_J \to 0$, we have \begin{align*} Tr(T_{w_\G x^{n+m n_0}}, M) &=Tr(T_{w_\G x^n} T_{t^{m n_0 \nu_x-\mu}} T_{t^{\mu}}, M_J)=Tr(T_{w_\G x^{n+m n_0}}, M_J) \\ &=Tr(T^J_{w_\G x^{n+m n_0}}, M_J).\end{align*}

Notice that $T^J_{w_\G x^{n+m n_0}}=T^J_{w_\G} (T^J_x)^{n+m n_0}=(T^J_x)^{n+m n_0} T^J_{w_\G}$. Since $0 \to \ker(T_{w_\G}: M_J \to M_J) \to M_J \to M_{J, \G} \to 0,$ we have $Tr(T^J_{w_\G x^{n+m n_0}}, M_J)=Tr((T^J_x)^{n+m n_0} T^J_{w_\G}, M_J)=(-1)^{\ell_J(w_\G)} Tr((T_x^J)^n, M_{J, \G})$.

By Proposition \ref{power}, $Tr(T_{\tw}^n, M)=(-1)^{n \ell(\tw)-n \ell(x)} Tr((T_x^{J_{\nu_x}})^n, M_{J, \G})$.

The ``in particular'' part follows from the proof of Corollary \ref{trace}.
\end{proof}

The following result is proved by Ollivier in \cite[Proposition 5.2]{O10}.

\begin{lem}
Let $J \subset F_0$ and $M \in R(\tilde \ch_{J, 0})_\kk$. Then $\tilde \ch_0 \otimes_{\tilde \ch_{J, 0}^+} M \cong \oplus_{d \in W_0^J} T_d \otimes M \cong \oplus_{d \in W_0^J} {}^\iota T_d \otimes M$ as vector spaces.
\end{lem}

\begin{cor}
Let $J_1 \subset J_2 \subset F_0$ and $M \in R(\tilde \ch_{J_1, 0})_\kk$. Then $$\tilde \ch_0 \otimes_{\tilde \ch_{J_2, 0}^+} (\tilde \ch_{J_2, 0} \otimes_{\tilde \ch_{J_1, 0}^+} M) \cong \tilde \ch_0 \otimes_{\tilde \ch_{J_1, 0}^+} M.$$
\end{cor}

\subsection{} \label{constr} Inspired by Lemma \ref{4.1},  we construct some representations of $\tilde \ch_0$.

For $J \subset F_0$ and $\G \subset J_{\aff}$, we set $\tilde \ch_{J, 0}(\G)=\ch_{J, 0} \rtimes \Omega_J(\G)$. This is the subalgebra of $\tch_{J, 0}$ generated by $T^J_s$ for $s \in J_{\aff}$ and $\Omega_J(\G)$.

Let $\chi \in \Omega_J(\G)^\vee=\Hom_\ZZ(\Omega_J(\G), \kk^\times)$. We extend $\chi$ as the $1$-dimensional $\tilde \ch_{J, 0}(\G)$-module, where $T^J_s$ acts by $-1$ if $s \in \G$ and by $0$ if $s \in J_{\aff} \smallsetminus \G$. Set $$\pi_{J, \G, \chi}=\tilde \ch_0 \otimes_{\tch_{J, 0}^+} (\tilde \ch_{J, 0} \otimes_{\tilde \ch_{J, 0}(\G)} \chi).$$


\subsection{} Let $\aleph=\{(J, \G); J \subset F_0, \G \subset J_{\aff}, \sharp W_\G < +\infty\}$. We define an equivalence relation $\sim$ and a partial order $<$ on $\aleph$ as follows. Let $(J, \G), (J', \G') \in \aleph$. We say that $(J, \G) \sim (J', \G')$ if $J=J'$ and $\G'=\t \G \t \i$ for some $\t \in \Omega_J$. We say that $(J, \G)<(J', \G')$ either $J \subsetneqq J'$ or $J=J'$ and $\G \supsetneqq \t \G' \t \i$ for some $\t \in \Omega_J$.

It is easy to see that for $(J, \G) \in \aleph$, $\chi \in \Omega_J(\G)^\vee$ and $\t \in \Omega_J$, we have $\Omega_J(\t \G \t \i)=\t \Omega_J(\G) \t \i$ and $\chi \circ \Ad(\t \i) \in \Omega_J(\t \G \t \i)^\vee$. Moreover, $\pi_{J, \G, \chi}$ and $\pi_{J, \G', \chi \circ \Ad(\t \i)}$ are isomorphic as $\tilde \ch_0$-modules.

\




The main result of this section is

\begin{thm}\label{char}
Let $(J, \G) \in \aleph$ and $\chi \in \Omega_J(\G)^\vee$. Let $\tw \in \tW_{\min}$ with associated standard pair $(x, \G')$. Then

(1) If $J \not \subset J_{\nu_x}$, then $Tr(T_{\tw}, \pi_{J, \G, \chi})=0$.


(2) If $J=J_{\nu_x}$ and $x \notin \Omega_J(\G)$, then $Tr(T_{\tw}, \pi_{J, \G, \chi})=0$.

(3) If $J=J_{\nu_x}$ and $x \in \Omega_J(\G)$, then $$Tr(T_{\tw}, \pi_{J, \G, \chi})=(-1)^{\ell(\tw)-\ell(x)} \chi(x) \sharp\{\t \in \Omega_J/\Omega_J(\G); \t \i \G' \t \subset \G\}.$$
\end{thm}

\begin{proof}
Set $J'=J_{\nu_x}$ and $M=\tilde \ch_{J, 0} \otimes_{\tilde \ch_{J, 0}(\G)} \chi$. Then $\pi_{J, \G, \chi} \cong \oplus_{d \in W_0^J} {}^\iota T_d \otimes M$ as vector spaces and $M=\oplus_{\t \in \Omega_J/\Omega_J(\G)}\kk T_\t^J \otimes v$ for any nonzero vector $v$ in the $1$-dimensional representation $\chi$ of $\tilde \ch_{J, 0}(\G)$. By Proposition \ref{power}, to compute the character of $T_{\tw}$, it suffices to compute the character of $T_{w_\G x^n}$ for $n \gg 0$.

Let $n_1 \in \NN$ such that $n_1 \nu_x \in \ZZ R_J$. Set $\l=n_1 \nu_x$.

(1) We first consider the case where $J \not \subset J'$. By Proposition \ref{u-x} (2), for $n \gg 0$, $T_{w_{\G'} x^{n+n_1}}=T_{w_{\G'} x^n} T_{t^\l}$. By definition, for $s \in S_0$, $$T_{t^\l} {}^\iota T_s=\begin{cases} {}^\iota T_s T_{t^\l}, & \text{ if } s \in J'; \\ 0, & \text{ otherwise.} \end{cases}$$

Thus for $d \in W_0^J$, \[\tag{a} T_{t^\l} {}^\iota T_d=\begin{cases} {}^\iota T_d T_{t^\l}, & \text{ if } d \in W_{J'}; \\ 0, & \text{ otherwise.} \end{cases}\] Moreover, in $M$ we have $T_{t^\l} (T_\t^J \otimes v)=T^J_{t^\l} (T_\t^J \otimes v)=T_\t^J \otimes (T^J_{\t\i t^\l \t} v)$.

By definition, for $\tu \in (W_J)_{\aff} \rtimes \Om_J$, $T^J_{\tu} v \neq 0$ if and only if $\tu \in W_\G \rtimes \Omega_J(\G)$. In particular, $\<\nu_{\tu}, \a^\vee\>=0$ for any $\a \in R_J$. Note that $\nu_{\t\i t^\l \t}=w(\nu_{t^\l})=n_1 w(\nu_{x})$ for some $w \in W_J$. Since $J \not \subset J'$, there exists $\b \in R_J$ such that $\<\nu_{x}, \b^\vee\> \neq 0$. Therefore $\<\nu_{\t\i t^\l \t}, w(\b^\vee)\> \neq 0$ and $T^J_{\t\i t^\l \t} v=0$. Hence $Tr(T_{w_\G x^{n+n_1}}, \pi_{J, \G, \chi})=0$ for $n \gg 0$. By Proposition \ref{power}, $Tr(T_{\tw}^{n+n_1}, \pi_{J, \G, \chi})=0$ for $n \gg 0$. By Corollary \ref{trace}, $Tr(T_{\tw}, \pi_{J, \G, \chi})=0$.

(2) Now we consider the case where $J=J'$. By (a), for $n \gg 0$, $T_{w_{\G'} x^n} \pi_{J,\G,\chi} \subset M$. Applying Lemma \ref{4.1}, we have $Tr(T_{\tw}, \pi_{J, \G, \chi})=(-1)^{\ell(\tw)-\ell(x)} Tr(T^J_x, T_{w_{\G'}}^J M)$.

We fix a representative for each coset $\Omega_J/\Omega_J(\G)$. Then $\{T_\t^J \otimes v; \t \in \Omega_J/\Omega_J(\G)\}$ is a basis of $M$. For $x \in \Omega_J$, the action of $T^J_x$ on $M$ permutes the lines $\kk(T_\t^J \otimes v)$ with $\t \in \Omega_J/\Omega_J(\G)$. Moreover, the action of $T^J_{w_{\G'}}$ stabilizes each line $\kk(T_\t^J \otimes v)$. If $x \notin \Omega_J(\G)$, then there is no line $\kk(T_\t^J \otimes v)$ stabilized by $T_x^J$ since $\Omega_J$ is abelian. Hence $Tr(T_{\tw}, \pi_{J, \G, \chi})=0$ in this case.

If $x \in \Omega_J(\G)$, then for any $\t \in \Omega_J/\Omega_J(\G)$, $T^J_x (T_\t^J \otimes v)=\chi(x) T_\t^J \otimes v$ and $$T_{w_{\G'}}^J (T_\t^J \otimes v)=T_{\t}^J \otimes T_{w_{\t \i \G' \t}}^J v=\begin{cases} (-1)^{\ell_J(w_{\G'})} T_{\t}^J \otimes v, & \text{ if } \t \i \G' \t \subset \G; \\ 0, & \text{ otherwise}. \end{cases}$$
Therefore, $\dim (T_{w_{\G'}}^J M)=\sharp \{\t \in \Omega_J/\Omega_J(\G); \t \i \G' \t \subset \G\}$ and
\begin{align*} Tr(T_{\tw}, \pi_{J, \G, \chi}) &=(-1)^{\ell(\tw)-\ell(x)} \chi(x) \dim (T_{w_{\G'}}^J M) \\ &=(-1)^{\ell(\tw)-\ell(x)} \chi(x) \sharp\{\t \in \Omega_J/\Omega_J(\G); \t \i \G' \t \subset \G\}.\end{align*} The proof is finished.
\end{proof}

\begin{cor}\label{nonzero}
Let $(J, \G) \in \aleph$ and $\chi \in \Omega_J(\G)^\vee$. Let $\tw \in \tW_{\min}$ with a standard pair $(x', \G')$. If $Tr(T_{\tw}, \pi_{J, \G, \chi}) \neq 0$, then $(J, \G) \le (J_{\nu_{x'}}, \G')$.
\end{cor}

\section{Representations of $\tilde \ch_0$}

Now we prove Theorem \ref{main'}.

\subsection{} We first show that $\{\pi_{J, \G, \chi}; (J, \G) \in \aleph/\sim, \chi \in \Omega_J(\G)^\vee\}$ is linearly independent in $R(\tch_0)_\kk$.

Suppose that $\sum_{(J, \G, \chi)} a_{J, \G, \chi} \pi_{J, \G, \chi}=0$ for some $a_{J, \G, \chi} \in \ZZ$.

Let $(J_1, \G_1) \in \aleph/\sim$ be a minimal element such that $a_{J, \G, \chi} \neq 0$ for some $\chi \in \Omega_J(\G)^\vee$. Set $\Omega_J(\G)_+=\{x \in \Omega_J; \nu_x \in X^+_\QQ\}$. It is easy to see that $\Omega_J(\G)_+$ generates $\Omega_J(\G)$.

By Theorem \ref{char}, for any $n \gg 0$ and $x \in \Omega_{J_1}(\G_1)_+$, $$\sum_{(J, \G, \chi)} a_{J, \G, \chi} Tr(T_{w_{\G_1} x^n}, \pi_{J, \G, \chi})=\sum_{\chi \in \Omega_{J_1}(\G_1)^\vee} a_{J_1, \G_1, \chi} (-1)^{n\ell(\tw)-n\ell(x)} \chi^n(x)=0.$$

Therefore $\sum_{\chi \in \Omega_{J_1}(\G_1)^\vee} a_{J_1, \G_1, \chi} \chi(x)=0$. By Dedekind's lemma, $a_{J_1, \G_1, \chi}=0$ for all $\chi \in \Omega_{J_1}(\G_1)^\vee$. That is a contradiction. Hence $a_{J, \G, \chi}=0$ for all $(J, \G, \chi)$.

\subsection{} \label{span} Next we show that $\{\pi_{J, \G, \chi}; (J, \G) \in \aleph/\sim, \chi \in \Omega_J(\G)^\vee\}$ spans $R(\tilde \ch_0)_\kk$.

For any $M \in R(\tilde \ch_0)_\kk$, let $\aleph(M)$ be the set of pairs $(J_{\nu_x}, \G)$ in $\aleph/\sim$ such that $Tr(T_{\tw}, M) \neq 0$ for some $\tw \in \tW_{\min}$ with an associated standard pair  $(x, \G)$.

We argue by induction on minimal elements in $\aleph(M)$.

If $\aleph(M)=\emptyset$, then $Tr(T_{\tw}, M)=0$ for all $\tw \in \tW_{\min}$. By Theorem \ref{basis}, $Tr(h, M)=0$ for all $h \in \tilde \ch_0$. Hence $M=0$.

Now suppose that $\aleph(M) \neq \emptyset$. Let $(J, \G)$ be a minimal element in $\aleph(M)$. We regard $M_{J, \G}$ as a virtual $\Omega_J(\G)$-module. Therefore $M_{J, \G}=\sum_{\chi \in \Omega_J(\G)^\vee} a_\chi \chi$ for some $a_\chi \in \ZZ$. We write $U_{J, \G}$ for the $\tilde \ch_0$-module $\sum_{\chi \in \Omega_J(\G)^\vee} a_\chi \pi_{J, \G, \chi}$. By Lemma \ref{4.1} and Theorem \ref{char}, for any $\tw \in \tW_{\min}$ with an associated standard pair $(x, \G) \in \aleph / \sim$ such that $J_{\nu_x}=J$, we have
\[\tag{a} Tr(T_{\tw}, M)=(-1)^{\ell(\tw)-\ell(x)} Tr(x, M_{J, \G})=Tr(T_{\tw}, U_{J, \G}).\]

Let $(J_1, \G_1), \cdots, (J_r, \G_r)$ be the set of all minimal elements in $\aleph(M)$. Set $$M'=M-\sum_{i=1}^r U_{J_i, \G_i}.$$

By (a) and Corollary \ref{nonzero}, if $\tw' \in \tW_{\min}$, with an associated standard pair $(x', \G')$, satisfies $Tr(T_{\tw'}, M') \neq 0$, then $(J_{\nu_{x'}}, \G')>(J_i, \G_i)$ for some $i$. By inductive hypothesis, $M'$ is a linear combination of $\{\pi_{J, \G, \chi}; (J, \G) \in \aleph/\sim, \chi \in \Omega_J(\G)^\vee\}$. So $M$ is a linear combination of $\{\pi_{J, \G, \chi}; (J, \G) \in \aleph/\sim, \chi \in \Omega_J(\G)^\vee\}$.

\subsection{} Motivated by \cite{CH}, we introduce rigid modules of $\tilde \ch_0$. Recall that $T_{\Sigma}$ for $\Sigma \in \text{Cyc}(\tW_{\min})$, form a basis of $\overline{\tilde \ch_0}$.
Set \begin{gather*} \overline{\tilde \ch_0}^{rig}=\oplus_{\Sigma \in \text{Cyc}(\tW_{\min}), J_{\nu_\Sigma}=F_0} \ZZ T_{\Sigma}, \\ \overline{\tilde \ch_0}^{nrig}=\oplus_{\Sigma \in \text{Cyc}(\tW_{\min}), J_{\nu_\Sigma} \subsetneqq F_0} \ZZ T_{\Sigma}. \end{gather*}

We call $\overline{\tilde \ch_0}^{rig}$ the rigid part of the cocenter and $\overline{\tilde \ch_0}^{nrig}$ the non-rigid part of the cocenter.

Let $M \in R(\tilde \ch_0)_\kk$. We say $M$ is rigid if $Tr(\overline{\tilde \ch_0}^{nrig}, M)=0$.

\begin{prop} \label{rig}
Let $M \in R(\tilde \ch_0)_\kk$. Then $M$ is rigid if and only if $M \in \oplus_{(F_0, \G) \in \aleph / \sim, \chi \in \Om(\G)^\vee} \ZZ \pi_{F_0, \G, \chi}$.
\end{prop}
\begin{rmk}
By Clifford's theory, the $\tch_0$-modules $\pi_{F_0, \G, \chi}$ for $(F_0, \G) \in \aleph / \sim$ and $\chi \in \Om(\G)^\vee$ are distinct simple modules.

\end{rmk}

\begin{proof}
By Theorem \ref{char}, $\pi_{F_0, \G, \chi}$ is rigid. On the other hand, assume $M=\sum_{(J,\G) \in \aleph/\sim, \chi \in \Om_J(\G)^\vee} a_{J, \G, \chi} \pi_{J, \G, \chi}$ with each $a_{J, \G, \chi} \in \ZZ$. Let $(J', \G')$ be a minimal element such that $a_{J', \G', \chi'} \neq 0$ for some $\chi' \in \Om_{J'}(\G')^\vee$. By the same argument as in \S\ref{span}, we see that $Tr(T_{w_{\G'} t^\mu}, M) \neq 0$ for some $\mu \in X^+(J')$. Hence $M$ is nonrigid unless $J'=F_0$, that is, $M \in \oplus_{(F_0, \G) \in \aleph / \sim, \chi \in \Om(\G)^\vee} \ZZ \pi_{F_0, \G, \chi}$.

\end{proof}

\subsection{} Let $\tw=w t^\l \in \tW$ with $\l \in X$ and $w \in W_0$. Write $\l=\mu_1-\mu_2$ with $\mu_1, \mu_2 \in X^+$. Following Vign\'{e}ras, we define $$E_{\tw}=q^{\frac{1}{2}(\ell(\mu_2)-\ell(\mu_1)-\ell(w)+\ell(\tw))} T_{w t^{\mu_1}}T_{t^{\mu_2}}\i \in \tch_q,$$ which dose not depend on the choices of $\mu_1$ and $\mu_2$. We still denote by $E_{\tw}$ its image in $\tch_0$. By \cite{V05}, the set $\{E_{\tw}; \tw \in \tW\}$ forms a basis of $\tch_0$.

\begin{lem}\label{bound}
Let $x, y \in \tW$ with $\ell(x) \le \ell(y)$. Then
\begin{gather*} q^{\frac{1}{2}(\ell(x)-\ell(y)+\ell(yx))} T_y T_{x\i}\i \in \bigl(\oplus_{z \in \tW, \ell(z) \ge \frac{1}{2}(\ell(y)-\ell(x)+\ell(yx))} \ZZ T_z \bigr)+q \ZZ[q] \tch_q, \\
q^{\frac{1}{2}(\ell(y)-\ell(x)+\ell(xy))} T_x T_{y\i}\i \in \bigl(\oplus_{z \in \tW, \ell(z) \ge \frac{1}{2}(\ell(y)-\ell(x)+\ell(xy))} \ZZ {}^\iota T_z \bigr)+q \ZZ[q] \tch_q.
\end{gather*}
\end{lem}
\begin{proof}
We prove the first statement. The second one can be proved in the same way.

We argue by induction on $\ell(x)$. If $\ell(x)=0$, then statement is obvious. Assume $\ell(x) \ge 1$ and the statement holds for any $x'$ with $\ell(x') < \ell(x)$. Let $s \in S_{\aff}$ such that $s x < x$.

If $y s < y$, then $$q^{\frac{1}{2}(\ell(x)-\ell(y)+\ell(yx))} T_y T_{x\i}\i=q^{\frac{1}{2}(\ell(sx)-\ell(ys)+\ell(ys sx))}T_{ys} T_{(sx)\i}\i$$ and $\ell(y)-\ell(x)+\ell(yx)=\ell(ys)-\ell(sx)+\ell(yssx)$. The statement follows from induction hypothesis.

If $y s >y$, then
\begin{align*} q^{\frac{1}{2}(\ell(x)-\ell(y)+\ell(yx))} T_y T_{x\i}\i= &q^{\frac{1}{2}(\ell(sx)-\ell(ys)+\ell(yx))} T_{ys} T_{(sx)\i}\i \\ &+q^{\frac{1}{2}(\ell(sx)-\ell(y)+\ell(yx)-1)}(1-q) T_y T_{(sx)\i}\i .
\end{align*}

By inductive hypothesis, $$q^{\frac{1}{2}(\ell(sx)-\ell(ys)+\ell(yx))} T_{ys} T_{(sx)\i}\i \in \bigl(\oplus_{z \in \tW, \ell(z) \ge \frac{1}{2}(\ell(y)-\ell(x)+\ell(xy))} \ZZ T_z \bigr)+q \ZZ[q] \tch_q.$$

Let $\a$ be the simple root associated to $s$ and $\b=x\i(\a)$. Then $\b<0$ since $s x <x$ and $yx(\b)=y(\a)>0$ since $y s>y$. Hence $y s x= y x s_{\b}<y x$. Therefore, $\ell(y s x) \le \ell(y x)-1$ and $\ell(sx)-\ell(y)+\ell(yx)-1 \ge \ell(sx)-\ell(y)+\ell(ysx)$.

If $\ell(y s x)<\ell(y x)-1$, then $\ell(sx)-\ell(y)+\ell(yx)-1>\ell(sx)-\ell(y)+\ell(ysx)$ and by inductive hypothesis, $q^{\frac{1}{2}(\ell(sx)-\ell(y)+\ell(yx)-1)}(1-q) T_y T_{(sx)\i}\i \in q \ZZ[q] \tch_q$ and the statement holds in this case.

If $\ell(y s x)=\ell(y x)-1$, then $\ell(y)-\ell(x)+\ell(yx) = \ell(y)-\ell(sx)+\ell(ysx)$ and by inductive hypothesis, \begin{align*} q^{\frac{1}{2}(\ell(sx)-\ell(y)+\ell(yx)-1)}(1-q) T_y T_{(sx)\i}\i & \in \bigl(\oplus_{z \in \tW, \ell(z) \ge \frac{1}{2}(\ell(y)-\ell(sx)+\ell(y s x))} \ZZ T_z \bigr)+q \ZZ[q] \tch_q \\ &=\bigl(\oplus_{z \in \tW, \ell(z) \ge \frac{1}{2}(\ell(y)-\ell(x)+\ell(yx))} \ZZ T_z \bigr)+q \ZZ[q] \tch_q
\end{align*}
The statement also holds in this case.
\end{proof}

\begin{cor}\label{bound1}
Let $\G \subset S_{\aff}$ with $\sharp W_{\G}<\infty$ and $\tw \in \tW$ with $\ell(\tw)> 2 \sharp W_\G$. Then in $\tch_0$,

$E_{\tw} \in \oplus_{z \in \tW, \supp(z) \nsubseteq \G} \ZZ T_z$ or $E_{\tw} \in \oplus_{z \in \tW, \supp(z) \nsubseteq \G} \ZZ {}^{\iota} T_z$.
\end{cor}
\begin{proof}
By definition, $E_{\tw}=q^{\frac{1}{2}(\ell(x)-\ell(y)+\ell(yx))} T_y T_{x\i}\i$ for some $x, y \in \tW$ such that $yx=\tw$. Applying Lemma \ref{bound}, we see that $E_{\tw} \in \oplus_{z \in \tW, \ell(z) > \sharp W_\G} \ZZ T_z$ if $\ell(x) \le \ell(y)$ and $E_{\tw} \in \oplus_{z \in \tW, \ell(z) > \sharp W_\G} \ZZ {}^{\iota} T_z$ if $\ell(y) \le \ell(x)$. The statement follows by noticing that $\supp(z) \nsubseteq \G$ if $\ell(z)>\sharp W_\G$.
\end{proof}

\


\begin{prop}\label{5.4}
Let $M \in R(\tilde \ch_0)$. The following conditions are equivalent:

(1) $E_{\tw} M=0$ for $\tw \in \tW$ with $\ell(\tw) \gg 0$.

(2) $Tr(\overline{\tilde \ch_0}^{nss}, M)=0$, where $\overline{\tilde \ch_0}^{nss}=\overline{\tilde \ch_0}^{nrig}+\iota(\overline{\tilde \ch_0}^{nrig})$.

(3) $M \in \oplus_{(F_0, \G), (F_0, F_0 \smallsetminus \G) \in \aleph, \chi \in \Omega_{F_0}(\G)^\vee} \ZZ \pi_{F_0, \G, \chi}$.
\end{prop}

\begin{rmk}
Condition (1) is the one of the equivalent definitions of supersingular modules due to Ollivier \cite[Proposition 5.4]{O13} and Vign\'eras \cite[Definition 6.10]{V14}. The equivalence between (1) and (3) was also proved in \cite[Theorem 5.14]{O13} and in \cite[Theorem 6.18]{V14}.
\end{rmk}

\begin{proof}
(1) $\Rightarrow$ (2). Let $\tw \in \tW_{\min}$ such that $J_{\nu_{\tw}} \subsetneqq F_0$. Let $(x, \G)$ be a standard pair associated to $\tw$. Choose $n_0, m_0 \in \ZZ_{>0}$ such that $n_0 \nu_x \in X^+$ and $w_\G x^{m_0} \in \tW_J^+$. Then $\ell (w_\G x^{m_0+r+k n_0}) = \ell(w_\G x^{m_0+r})+ \ell(x^{k n_0})$ for $k \in \NN$ and $0 \le r \le n_0-1$. Thus $T_{w_\G x^{m_0+r+k n_0}}=T_{w_\G x^{m_0+r}} E_{x^{k n_0}}$. By assumption, we have $T_{w_\G x^n} M=0$ for $n \gg 0$. Applying Proposition \ref{power}, $Tr(T_{\tw}^n, M) = \pm Tr(T_{w_\G x^n}, M)=0$ and hence $Tr(T_{\tw}, M)=0$. The equality $Tr({}^\iota T_{\tw\i}, M)=0$ follows in a similar way by noticing that ${}^\iota T_{x^{-k n_0}}=E_{x^{-k n_0}}$ for $k \in \NN$.

(2) $\Rightarrow$ (3). By Proposition \ref{rig}, $M$ and its pullback ${}^\iota M$ via $\iota$ lie in the $\ZZ$-span of $\{\pi_{F_0, \G, \chi}; (F_0, \G) \in \aleph, \chi \in \Om(\G)^\vee\}$. By definition ${}^\iota \pi_{F_0, \G, \chi}=\pi_{F_0, S_{\aff} \smallsetminus \G, \chi}$. Thus $M$ also lies in the $\ZZ$-span of $\{\pi_{F_0, \G, \chi}; (F_0, \G \subset S_{\aff}) \in \aleph, \chi \in \Om(\G)^\vee\}$. Therefore, $M$ lies in the $\ZZ$-span of $\{\pi_{F_0, \G, \chi}; (F_0, \G), (F_0, F_0 \smallsetminus \G) \in \aleph, \chi \in \Om(\G)^\vee\}$.

(3) $\Rightarrow$ (1). Let $\G \subset S_{\aff}$. By definition, $T_x \pi_{F_0, \G, \chi}={}^{\iota} T_x \pi_{F_0, \G, \chi}=0$ for any $x \in \tW$ such that $\supp(x) \nsubseteq \G$ and $\supp(x) \nsubseteq S_{\aff} \smallsetminus \G$. Assume $\sharp W_\G, \sharp W_{S_{\aff} \smallsetminus \G} < +\infty$. Applying Corollary \ref{bound1}, $E_{\tw} \pi_{F_0, \G, \chi}=0$ for $\tw \in \tW$ with $\ell(\tw)> 2\sharp W_\G, 2\sharp W_{S_{\aff} \smallsetminus \G}$.
\end{proof}

\section*{Acknowledgement}
The first-named author was introduced to the representations of affine $0$-Hecke algebras by Marie-France Vign\'eras, who explained the beauty and importance of supersingular modules and encouraged the author to apply the method in \cite{CH} to the study of affine $0$-Hecke algebras. It is a great pleasure to thank her. The authors also would like to thank Dan Ciubotaru and George Lusztig for many useful discussions on affine Hecke algebras, and to thank Noriyuki Abe for sending us some lecture notes on modular Iwahori-Hecke algebras.

\end{document}